\documentclass[11pt]{amsart}
\usepackage{amssymb,amscd, graphicx, epsf}

\begin{document}

\newtheorem{defi}{Definition}
\newtheorem*{thm}{Theorem}
\newtheorem*{conjecture}{Conjecture}
\newtheorem{prop}{Proposition}
\newtheorem*{rem}{Remark}
\newtheorem{numrem}{Remark}
\newtheorem{lemma}{Lemma}
\newtheorem{coro}{Corollary}
\newtheorem{quest}{Question}
\newtheorem{corollary}{Corollary}

\title[Packing and covering densities]{The set of packing and covering densities\\ of convex disks}

\author{W{\l}odzimierz Kuperberg}
\address{W.~Kuperberg, Mathematics \& Statistics, Auburn University, Auburn, AL 36849-5310, USA}
\email{kuperwl@auburn.edu}


\begin{abstract}   For every convex disk $K$ (a convex compact subset of the plane,
with non-void interior), the packing density $\delta(K)$ and covering density
$\vartheta(K)$ form an ordered pair of real numbers, {\em i.e.}, a point in ${\mathbb R}^2$.
The set $\Omega$ consisting of points assigned this way to all convex disks is the
subject of this article. A few known inequalities on $\delta(K)$ and $\vartheta(K)$
jointly outline a relatively small convex polygon $P$ that contains $\Omega$, while
the exact shape of $\Omega$ remains a mystery.  Here we describe explicitly a leaf-shaped
convex region $\Lambda$ contained in $\Omega$ and occupying a good portion of $P$. 
The sets $\Omega_T$ and $\Omega_L$ of translational packing and covering densities
and lattice packing and covering densities are defined similarly, restricting the allowed
arrangements of $K$ to translated copies or lattice arrangements, respectively.  Due to
affine invariance of the translative and lattice density functions, the sets $\Omega_T$
and $\Omega_L$ are compact.  Furthermore, the sets $\Omega$,  $\Omega_T$ and
$\Omega_L$ contain the subsets $\Omega^\star$,  $\Omega_T^\star$ and $\Omega_L^\star$
respectively, corresponding to the centrally symmetric convex disks $K$, and our leaf $\Lambda$
is contained in each of $\Omega^\star$,  $\Omega_T^\star$ and $\Omega_L^\star$.

\end{abstract}

\maketitle

\section{Introduction: Definitions and Notation.}

An $n$-dimensional {\em convex body} is a compact, convex subset of ${\mathbb R}^n$
that contains an interior point.  A $2$-dimensional convex body is called a {\em convex disk}.  The convex hull of
a set $S$ is denoted by ${\rm Conv}(S)$.

A family $\mathcal F=\{K_i\}$ of subsets  of ${\mathbb R}^n$, each congruent to a given convex
body $K$, is a {\em packing} (of ${\mathbb R}^n$ with copies of $K$) if their interiors are mutually
disjoint, and it is a {\em covering} if their union is ${\mathbb R}^n$.  If $\mathcal F$ is a packing and
a covering, then it is called a {\em tiling}, and we say that such a convex body $K$
{\em admits a tiling of ${\mathbb R}^n$} or that $K$ is a {\em tile}, for brevity.

For a (measurable) set $S$ in ${\mathbb R}^n$, $|S|$ denotes the $n$-dimensional measure of $S$,
and for a family of sets $\mathcal F=\{K_i\}$ and a region $D\subset{\mathbb R}^n$, we will denote by
$|{\mathcal F}\cap D|$ the sum of the volumes of $K_i\cap D$ over all $K_i\in {\mathcal F}$.
Let $B_r$ denote the ball of radius $r$ in ${\mathbb R}^n$, centered at the origin.

Given any family ${\mathcal F}=\{K_i\}$ of congruent copies of a convex body $K$ (in particular, a packing
or a covering), the {\em density} of ${\mathcal F}$ is defined as
$$
d({\mathcal F})=\lim_{r\to\infty}\frac{|{\mathcal F}\cap{B}_r|}{|{B}_r|}\,.
$$

Here, if the limit does not exist, then in case ${\mathcal F}$ is a packing we take the upper limit, {\em limsup},
and in case it is a covering we take the lower limit, {\em liminf}, in place of the limit.

For every convex body $K$, the supremum of the set of densities $d({\mathcal F})$ taken over all packing arrangements
with copies of $K$ is called {\em the packing density}  of $K$ and is denoted by $\delta(K)$. Similarly, the infimum of the
set of densities $d({\mathcal F})$ taken over all covering arrangements with copies of $K$ is called {\em the covering
density} of $K$ and is denoted by $\vartheta(K)$. 

Often certain restrictions on the structure of the arrangements ${\mathcal F}$ are imposed: one may consider
arrangements of translated copies of $K$ only, or just lattice arrangements of translates of $K$.  In these cases,
the corresponding densities assigned to $K$ by analogous definitions are: the {\em translative} packing and covering
density of $K$, denoted by $\delta_T(K)$ and $\vartheta_T(K)$, and the {\it lattice} packing and covering density of $K$,
denoted by $\delta_L(K)$ and $\vartheta_L(K)$, respectively.

It is easily seen that these quantities satisfy the following inequalities:
$$
0\le\delta_L(K)\le\delta_T(K)\le\delta(K)\le1\le\vartheta(K)\le\vartheta_T(K)\le\vartheta_L(K).
\eqno{(1.1)}$$

For more details on these notions and their basic properties see \cite{FTK}. Here we only mention the following: 
\vspace{8pt}
\begin{itemize}

\item[$(i)$] {\sl Each of the extreme densities $\delta(K)$,  $\delta_T(K)$,  $\delta_L(K)$, $\vartheta(K)$, $\vartheta_T(K)$,
and $\vartheta_L(K)$ is attained by a corresponding arrangement $\mathcal F$. In other words, for every convex body $K$
there is a maximum density packing and a minimum density covering in each of the three types: by arbitrary isometries,
by translates only, and by lattice arrangements only.}
\vspace{8pt}
\item[$(ii)$] {\sl Each of the six densities $\delta$,  $\delta_T$,  $\delta_L$, $\vartheta$, $\vartheta_T$, and $\vartheta_L$,
considered as a real-valued function defined on the hyperspace $\mathcal K^n$ of all $n$-dimensional convex bodies
furnished with the Hausdorff metric, is continuous.}
\vspace{8pt}
\item[$(iii)$] {\sl Each of the four densities $\delta_T$,  $\delta_L$, $\vartheta_T$, and $\vartheta_L$ is affine-invariant,
hence each of them can be viewed as a (continuous) function defined on the space $\left[\mathcal K^n \right]$ of affine
equivalence classes of all convex $n$-dimensional bodies---a quotient space of $\mathcal K^n$.}
\vspace{8pt}
\item[$(iv)$] {\sl $\delta(K)=1\Longleftrightarrow \vartheta(K)=1 \Longleftrightarrow$ $K$ is a tile.}
\end{itemize}

\section{The background.}

The main result, as outlined in the Abstract, concerns the packing and covering densities of convex disks. Therefore the brief survey given in this section is focused on the two-dimensional case.

The lattice packing density of the circular disk $B^2$ (the unit ball in ${\mathbb R}^2$)
has been determined by Lagrange \cite{JLL} in 1773: $\delta_L(B^2)=\pi/\sqrt{12}$.  In 1831, Gauss \cite{CFG}
established that $\delta_L(B^3)=\pi/\sqrt{12}$.  The first proof of the equality $\delta(B^2)=\pi/\sqrt{12}$ was given by
Thue \cite{THUE} in 1910.  The analogous result for the covering, $\vartheta(B^2)=2\pi/\sqrt{27}$, is due to
Kershner \cite{RK39}.

The following significant generalization of Thue's and Kershner's theorems was given by L.~Fejes T{\'o}th \cite{LFT53}.   For every convex disk $K$, let $h(K)$ and $H(K)$ denote the maximum area hexagon contained in $K$ and
the minimum area hexagon containing $K$, respectively. Then no packing of the plane with congruent copies of $K$ can be of density greater than ${|K|}/{|H(K)|}$.  Fejes T{\'o}th's proof of the bound of ${|K|}/{|h(K)|}$ for coverings requires that the copies $K_i$ and $K_j$ of $K$ {\em do not cross each other}, meaning that at least one of the sets $K_i\setminus K_J$ and  $K_j\setminus K_i$ is connected, hence the second result is not quite analogous to the first one.  However, since two translates of a convex disk cannot cross each other, the two results produce the following inequalities:
$$
\delta(K)\le\frac{|K|}{|H(K)|}\ \ {\rm and}\ \  \vartheta_T(K)\ge\frac{|K|}{|H(K)|}\,. \eqno{(2.1)}
$$

Along with the theorems of Dowker \cite{CD44} on inscribed and circumscribed polygons in centrally symmetric disks, the above inequalities imply that if  $K$ is centrally symmetric, then
$$
\delta(K)=\delta_T(K)=\delta_L(K)=\frac{|K|}{|H(K)|}\ \ {\rm and}\ \ \vartheta_T(K)=\vartheta_L(K)=\frac{|K|}{|h(K)|}\,. \eqno{(2.2)}
$$
These two equalities generalize Thue's and Kershner's theorems, respectively.

The equality $\vartheta(K)=\frac{|K|}{|h(K)|}$ for centrally symmetric disks $K$, conjectured by L.~Fejest T\'{o}th \cite{LFT50} (see also \cite{LFT53}) still remains unproven.

Chakerian and Lange \cite{RK39} proved that every convex disk $K$ is contained in a quadrilateral of area at most $\sqrt2\, |K|$. Since every quadrilateral admits a tiling of the plane, they established the inequality
$$\delta(K)\ge{\sqrt2}/2=0.7071\ldots\,. \eqno{(2.3)} $$

Henceforth, $K$ will denote an arbitrary (not necessarily centrally symmetric) convex disk, unless otherwise explicitly assumed.

The inequality of Chakerian and Lange was subsequently improved in \cite{KK90} to
$$\delta(K)\ge{\sqrt3}/2=0.8660\ldots\,. \eqno{(2.4)} $$

Concerning upper bounds on covering density, the inequality
$$\vartheta(K)\le 8(2\sqrt3-3)/3=1.2376043\ldots\,, \eqno{(2.5)}$$
established in \cite{WK89}, was improved by D.~Ismailescu \cite{DI98} to
$$\vartheta(K)\le1.2281772\ldots\,. \eqno{(2.6)}$$

The upper bound in the above inequality, though presently the best known, is not likely to be sharp.  However, for the centrally
symmetric convex disks $K$, the second equality in (2.2) and a theorem of Sas \cite{ES39} on maximum area polygons inscribed
in a convex disk imply that
$$\vartheta(K)\le2\pi/\sqrt{27}=1.209\ldots\,, \eqno{(2.7)}$$
which is sharp, since equality holds for the circular disk.

The following inequality, linking the packing density $\delta(K)$ and the covering density $\vartheta(K)$:
$$3\vartheta(K)\le4\delta(K) \eqno{(2.8)}$$
was proved in \cite{WK87}. The inequality is sharp, since equality holds in the case of the circular disk.

\section{Packing and covering via tiling.}

In this section we do not restrict ourselves to two dimensions.  Here all convex bodies are assumed to be $n$-dimensional, also packing and covering means packing and covering of $\mathbb{R}^n$.

We begin with the following simple observation.

Assume that a convex body $K$ is contained in a tile $T$. Then a space tiling with congruent replicas of $T$ naturally yields a space packing with congruent replicas of $K$. Similarly, if $K$ contains $T$, then the tiling yields a space covering with  replicas of $K$.  Such a packing, resp.~covering, is said to be {\em generated} by the tile $T$.

\begin{prop} The density of a packing generated by a tile $T$ containing $K$ is $|K|/|T|$. Similarly, the density of a covering generated by a tile $t$ contained in $K$ is $|K|/|t|$.\end{prop} 

We omit the easy, natural proof of this proposition.  Following are two simple, but quite useful statements.

\begin{prop} Suppose a densest space packing with congruent copies of $K$, that is, a packing of density $\delta(K)$, is generated by a tile $T$ containing $K$, and let $M$ be a convex body ``sandwiched'' between $K$ and $T$, that is, $K\subset M\subset T$. Then 
$\delta(M)=|M|/|T|$.
\end{prop}

Henceforth we will say that set $B$ is {\em sandwiched between} sets $A$ and $C$ to mean that $A\subset B\subset C$ or  $C\subset B\subset A$.

\begin{prop} Suppose a thinnest space covering with congruent copies of $K$, that is, a covering of density $\vartheta(K)$, is generated by a tile $T$ contained in $K$, and let $M$ be a convex body sandwiched between $K$ and $T$. Then $\vartheta(M)=|M|/|T|$.\end{prop}

\begin{proof}[Proof of Proposition 2]  Indeed, $T$ generates a packing with replicas of $M$ that is of density $|M|/|T|$, and a packing with higher density is impossible, since such a packing would yield a packing with copies of $K$ of density exceeding $\delta(K)$. \end{proof}

Proof of Proposition 3 is completely analogous to that of Proposition 2.

As an application of the above, we obtain the following two propositions.

\begin{prop} Let $K$ be a convex disk sandwiched between the circular disk $C$ and a regular hexagon $H$ circumscribing $C$ {\em (see Fig.~\ref{ThueKersh}{\em a})}.  Since, by Thue's theorem, $\delta(C)=|C|/|H|$ and since $H$ is a tile, we conclude that
$$\delta(K)=|K|/|H|.$$  Moreover, since the hexagon $H$ is centrally symmetric, it tiles the plane in a lattice-like manner, which implies that 
$$\delta(K)=\delta_T(K)=\delta_L(K).$$
\end{prop}

\begin{figure}[h]
\begin{center}
\includegraphics[scale=.6]{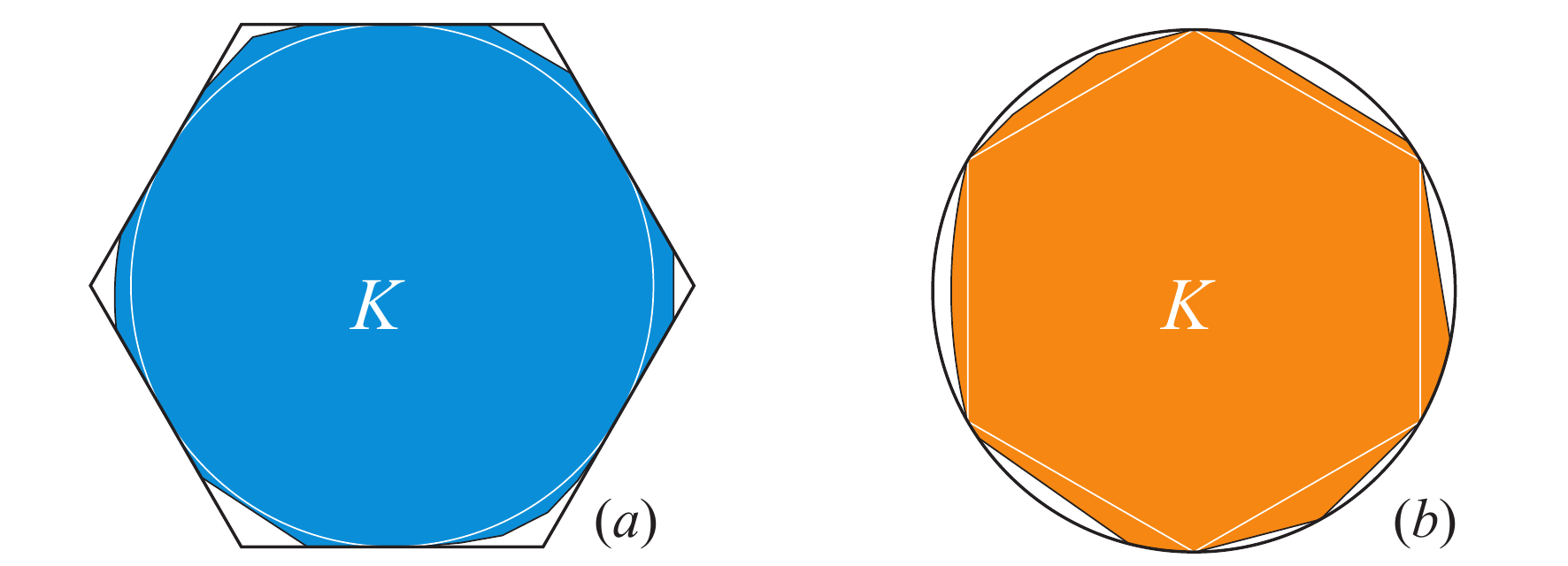}
\caption{Disk $K$ sandwiched between the circle and ({\em a}) the circumscribed, ({\em b}) the inscribed, regular hexagon.}
\label{ThueKersh}
\end{center}
\end{figure}

Similarly:
\begin{prop} Let $K$ be a convex disk sandwiched between the circular disk $C$ and a regular hexagon $h$ inscribed in $C$ {\em (see Fig.~\ref{ThueKersh}{\em b})}. By the above Propositions and by Kershner's theorem, we conclude that 
$$\vartheta(K)=|K|/|h|.$$
Moreover,
$$\vartheta(K)= \vartheta_T(K)= \vartheta_L(K).$$
\end{prop}

\begin{rem}
{\em  In view of the recent proof of the sphere packing conjecture in $\mathbb{R}^3$ (also known as the Kepler Conjecture) by Hales \cite{TH05}, the rhombic dodecahedron circumscribed about the ball $B^3$ generates a densest ball packing. Hence:}

For every convex body $K$ sandwiched between the ball and the rhombic dodecahedron circumscribed about it, the packing density $\delta(K)=\delta_T(K)=\delta_L(K)$ is the ratio between the volume of $K$ and the volume of the rhombic dodecahedron.
\end{rem}

\section{The subset $\Omega$ of the plane, consisting of pairs $\left(\delta(K),\vartheta(K)\right)$.}

Define $\omega: \mathcal{K}^2\to\mathbb{R}^2$ by $\omega(K)=\left(\delta(K),\vartheta(K)\right)$ for every $K\in\mathcal{K}^2$.  By continuity of each of the real-valued functions $\delta$ and $\vartheta$, the function $\omega$ is continuous. Let $\Omega=\omega\left(\mathcal{K}^2\right)$.  The inequalities (2.4), (2.6), (2.7), along with the obvious ones, $\delta(K)\le1$ and $\vartheta(K)\ge1$, can be interpreted as five half-planes in the $(x,y)$-plane $\mathbb{R}^2$, each containing the set $\Omega$, namely
$$(a)\ x\ge0.8660\ldots\,,\ \quad (b)\ x\le1,\quad (c)\ y\le1.2281772\ldots\,,\quad  (d)\ y\ge1,\quad {\rm and}\ (e)\ y\le\frac{4}{3}x\,.$$
 The intersection of these half-planes is a pentagon, we denote it $P$, containing $\Omega$, see Fig.~\ref{pent}.

\begin{figure}[h]
\begin{center}
\includegraphics[scale=.6]{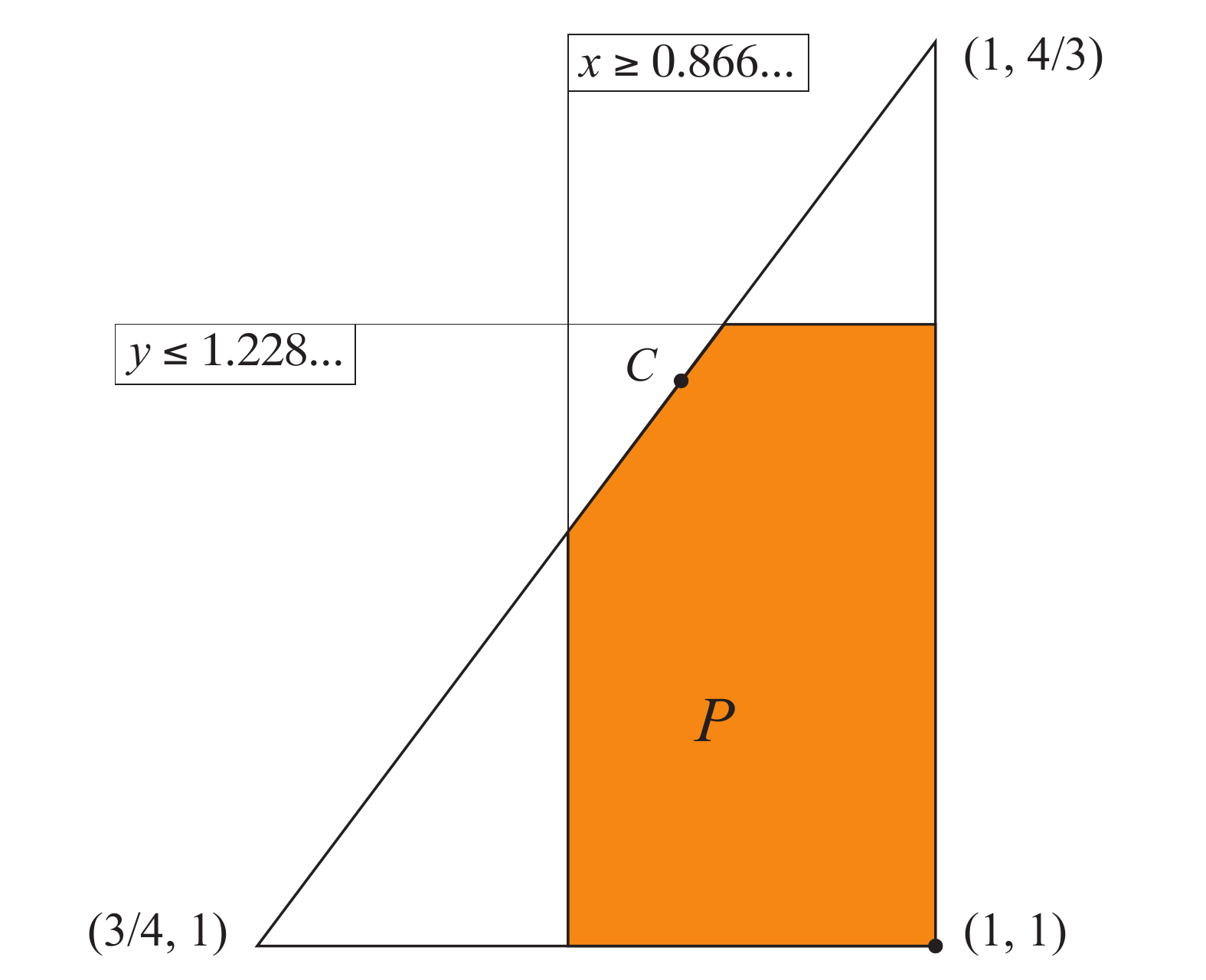}
\caption{Pentagon $P$ that contains the set $\Omega$.}
\label{pent}
\end{center}
\end{figure}

The vertex $(1,1)$ of $P$ corresponds to all plane-tiling polygons, and is the only point of $\Omega$ lying on the union of the two sides of $P$ containing this vertex.  The point $C=\left(\pi/\sqrt{12}, 2\pi/\sqrt{27}\,\right)$ lies on the slant side of $P$ and corresponds to the circular disk $B^2$, and perhaps also to every ellipse---we say {\em perhaps}, because in general the covering density of an ellipse is unknown.  It would be interesting to know if $C$ is the only point of $\Omega$ lying on the line $y=\frac{4}{3}x$, and also if $C=\omega(K)$ for some non-elliptical convex disk $K$.

\begin{numrem}
{\em The sets $\Omega_T$ and $\Omega _L$ are defined in a similar way as $\Omega$, just by replacing the function
$\omega=(\delta,\vartheta)$ by $\omega_T=(\delta_T,\vartheta_T)$ and by $\omega_L=(\delta_L,\vartheta_L)$, respectively.  Also, one can define the subset $\Omega^*$ of $\Omega$, by restricting the function $\omega$ to the subspace $\mathcal{K}^{*2}$ of $\mathcal{K}^2$ consisting of centrally symmetric disks $K$, and then define the sets $\Omega_T^*\subset\Omega_T$ and $\Omega_L^*\subset\Omega_L$ in the same way. Equalities (2.2) imply that $\Omega^*_T=\Omega^*_L$.

Of course, the corresponding six $\Omega$-type sets are readily defined in the very same way for
$n$-dimensional convex bodies with $n>2$.  However, already in dimension $3$, most of the properties
and inequalities analogous to those mentioned in the previous sections either do not hold or are not known
to be true.}
\end{numrem}

\begin{numrem}
{\em The {\em Minkowski sum} of convex bodies, defined by
$$K+M=\{x+y: x\in K, y\in M\},$$
is a continuous map from $\mathcal{K}^n\times\mathcal{K}^n$ to $\mathcal{K}^n$. Therefore the space $\mathcal{K}^n$ is {\em contractible}, which means that there is a homotopy between the identity on $\mathcal{K}^n$ and a constant map.  This implies that the space $\mathcal{K}^n$ is simply-connected.  This in turn implies that each of the sets $\Omega$, $\Omega_T$, and $\Omega_L$ is pathwise-connected, being a continuous image of a pathwise-connected space. But a continuous image of a simply-connected space need not be simply-connected.
However, if a simple closed curve $J$ in $\Omega$ (in $\Omega_T$ or in $\Omega_L$) is the image under $\omega$ ($\omega_T$, $\omega_L$, respectively) of a simple closed curve in $\mathcal{K}^n$, then $J$ is contractible to a point in $\Omega$ ($\Omega_T$, $\Omega_L$, resp.).  Such a contractible simple closed curve $J$ bounds a unique topological disk in $\mathcal{R}^2$, therefore the disk must be contained in $\Omega$ ($\Omega_T$, $\Omega_L$, resp.).  

Since the Minkowski sum of two centrally symmetric sets is a centrally symmetric set, the analogous statements hold for the space $\mathcal{K}^{*n}$ of centrally symmetric $n$-dimensional convex bodies, and to the corresponding sets $\Omega^*$ $\Omega^*_T$ and $\Omega^*_L$.}
\end{numrem}

\begin{figure}[h]
\begin{center}
\includegraphics[scale=.6]{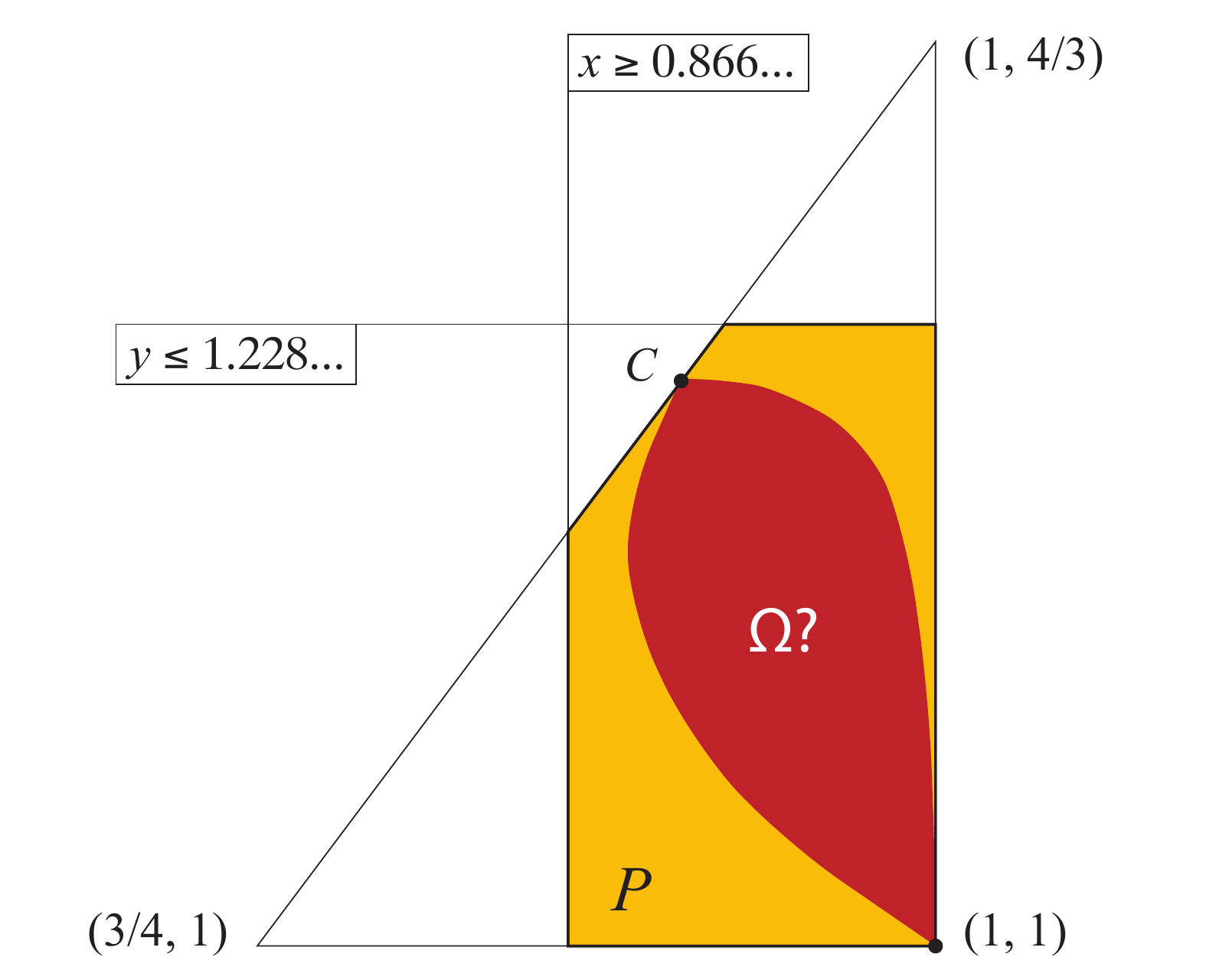}
\caption{The conjectured teardrop shape of $\Omega$.}
\label{omega}
\end{center}
\end{figure}

\begin{numrem}
{\em There are reasons to believe that the shape of $\Omega$ resembles that of a teardrop, as shown in Fig.~\ref{omega}. First, the lower bound $\delta(K)\ge\sqrt3/2=0.866\ldots$ is very likely possible to raise, perhaps fairly close to $0.9$. Also, the upper bound $\vartheta(K)\le1.228\ldots$ is likely possible to lower, perhaps all the way down to $2\pi/\sqrt{27}=1.209\ldots\,$. Then, since the vertex $(1.1)$ of $P$ is the only point that $\Omega$ has in common with each of the two sides of $P$ containing it, it seems reasonable to expect that if a convex disk can pack the plane very efficiently ({\em i.e.}, has packing density close to $1$), then it should cover the plane efficiently as well.  More specifically, we conjecture that there exist a convex-to-the-left curve in $P$ containing $(1.1)$ and separating $\Omega$ from the ray $y=1$, $x<1$, and also there is a convex-to-the-right curve in $P$, containing $(1.1)$ and separating $\Omega$ from the ray $x=1$, $y>1$. Possibly, the two curves may lie on the boundary of $\Omega$, meeting at $C$, and forming the teardrop shape between them.}
\end{numrem}

\section{Certain two arcs in $\Omega$ and the disk they enclose.}

Let $H$ be the regular hexagon of unit edge length let $D_t$ $(0\le t\le 1)$ be a circular disk shrinking monotonically and continuously in $t$, beginning with being circumscribed about $H$ and ending up inscribed in $H$. Let $\mathcal{A}$ be the arc in $\mathcal{K}^2$ consisting of the convex disks $K_t=D_t\cap H$, see Fig.~\ref{CirHex}.

\begin{figure}[h]
\begin{center}
\includegraphics[scale=.45]{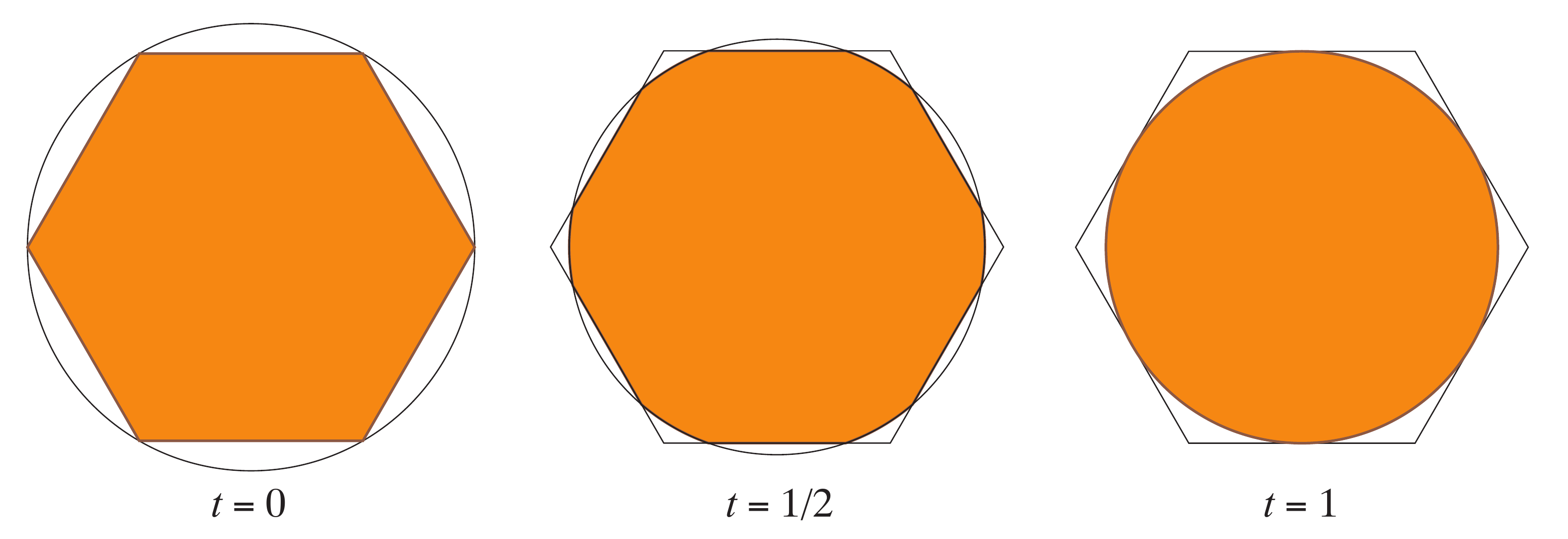}
\caption{The disks $K_t$ representing points of the arc $\mathcal{A}$ for $t=0,\ t=1/2$, and $t=1$.}
\label{CirHex}
\end{center}
\end{figure}

Similarly, let $\mathcal{B}$ be the arc in $\mathcal{K}^2$ consisting of the convex disks $L_t={\rm Conv}\left(D_t\cup H\right)$, see Fig.~\ref{HexCir}.

\begin{figure}[h]
\begin{center}
\includegraphics[scale=.45]{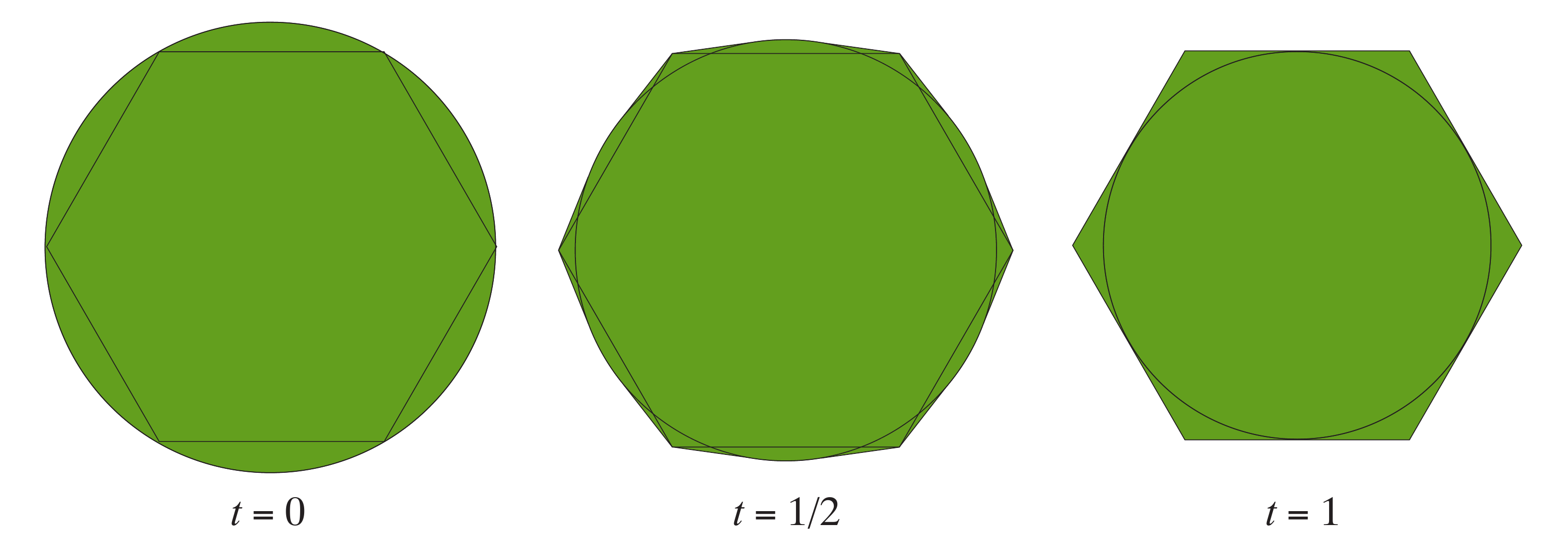}
\caption{The disks $L_t$ representing points of the arc $\mathcal{B}$ for $t=0,\ t=1/2$, and $t=1$.}
\label{HexCir}
\end{center}
\end{figure}

Each of the arcs $\mathcal{A}$ and $\mathcal{B}$ connects points $H$ and $D$, and $\mathcal{A}\cup\mathcal{B}$  is a simple closed curve in $\mathcal{K}^2$. Each of $\alpha=\omega(\mathcal{A})$ and $\beta=\omega(\mathcal{B})$ is a path in $\Omega$ joining point $(1,1)$ with point $C=(\pi/\sqrt{12}, 2\pi/\sqrt{27})$.

For a parametric description of $\alpha$ and $\beta$, we observe that  each disk $K_t$ is sandwiched between the circular disk $D_t$ and a regular hexagon inscribed in $D_t$, and also is sandwiched between the regular hehagon $H$ and the circle inscribed in it.  Similarly, each disk $L_t$ is sandwiched between the circular disk $D_t$ and a regular hexagon circumscribed about $D_t$, and also is sandwiched between the regular hehagon $H$ and the circle circumscribed about it. Using Propositions 4 and 5 and by computing the areas of the disks $K_t$ and $L_t$ and the disks between which they are sandwiched,  we obtain explicit parametric presentations of the paths $\alpha$ and $\beta$ as follows.

In our parameterization of the curve $\alpha=\omega(\mathcal{A})$, we choose the parameter $u$ associated with $\omega(K_t)$ to be the length of a rectilinear component of the boundary of $K_t$.  For the curve $\beta=\omega(\mathcal{B})$, we choose the parameter $v$ associated with $\omega(K_t)$ to be the total length of two rectilinear components of the boundary of $L_t$, see Fig.~\ref{param}.

\begin{figure}[h]
\begin{center}
\includegraphics[scale=.6]{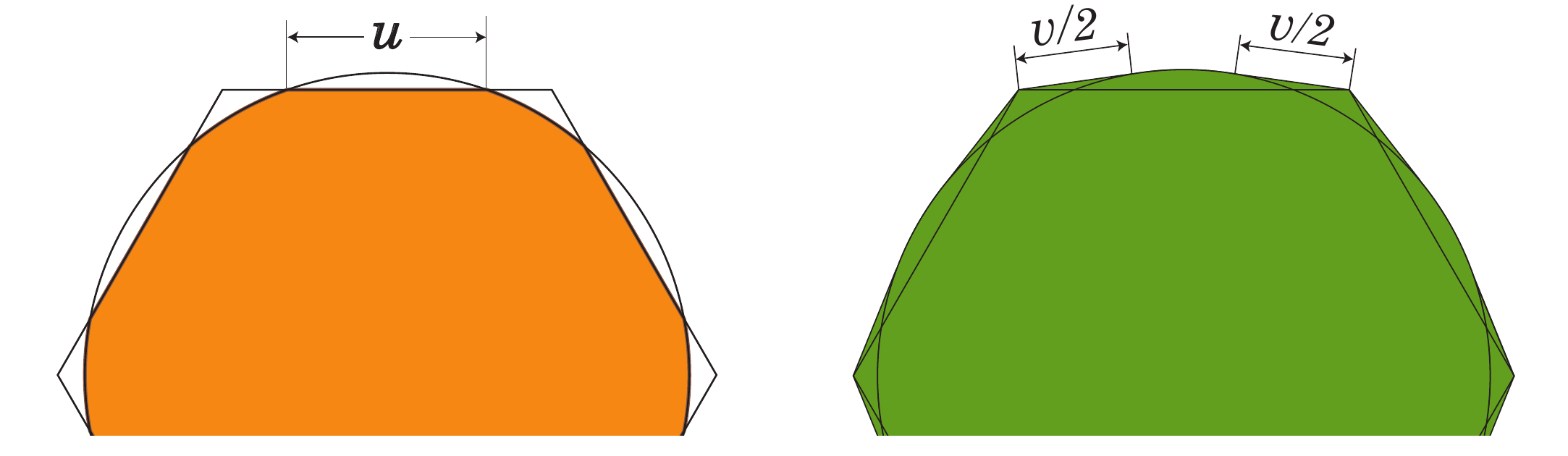}
\caption{The parameters $u$ and $v$.}
\label{param}
\end{center}
\end{figure}

By elementary computations of corresponding areas we obtain the formulae for the packing and covering densities $\delta(K_t)$, $\vartheta(K_t)$ expressed as functions of $u$, and $\delta(L_t)$ and $\vartheta(L_t)$ expressed as functions of $v$, that constitute the following parametric presentations of the arcs $\alpha$ and $\beta$:

\vspace{.2in}

\hspace{.5in}
$\alpha : \left\{
\begin{tabular}{lll}
$x=f_1(u)=\displaystyle\frac{2\pi-12\arcsin{(u/2)}+3\sqrt{4-u^2}}{\sqrt3(4-u^2)}\,;$  \\
{} & $(0\le u\le1)$ \\ \\
$y=g_1(u)=\displaystyle\frac{2\pi-12\arcsin{(u/2)}+3\sqrt{4-u^2}}{3\sqrt3}\,;$ \\
\end{tabular}
\right.
$

\vspace{.2in}

and

\vspace{.2in}

\hspace{.5in}
$\beta : \left\{
\begin{tabular}{lll}
$x=f_2(v)=\sqrt3\left(\displaystyle\frac{\pi}{6}-\arcsin{(v/2)}+\frac{v}{\sqrt{4-v^2}}\right)$;  \\
{} & $(0\le v\le1)$ \\
$y=g_2(v)=\displaystyle\frac{\sqrt3}{3}\left((4-v^2)\left(\frac{\pi}{6}-\arcsin{(v/2)}\right)+v\sqrt{4-v^2}\,\right)$.\\
\end{tabular}
\right.
$

\vspace{.2in}

Path $\alpha$ begins at $(1,1)$ and ends at $C=(\pi/\sqrt{12}, 2\pi\/\sqrt{27})$, and path $\beta$ begins at $C=(\pi/\sqrt{12}, 2\pi\/\sqrt{27})$ and ends at $(1,1)$, therefore $\alpha + \beta$ is a loop.  Since each of the functions $\psi_1$ and $\psi_2$ is strictly monotonic ($\psi_1$ decreases and $\psi_2$ increases), each of the paths $\alpha$ and $\beta$ is an arc. Moreover, $\alpha$ is concave, and $\beta$ is convex, hence the loop $\alpha+\beta$ is a simple closed curve that bounds a leaf-shaped convex disk $\Lambda$, see Fig.~\ref{leaf}.

\begin{figure}[h]
\begin{center}
\includegraphics[scale=.6]{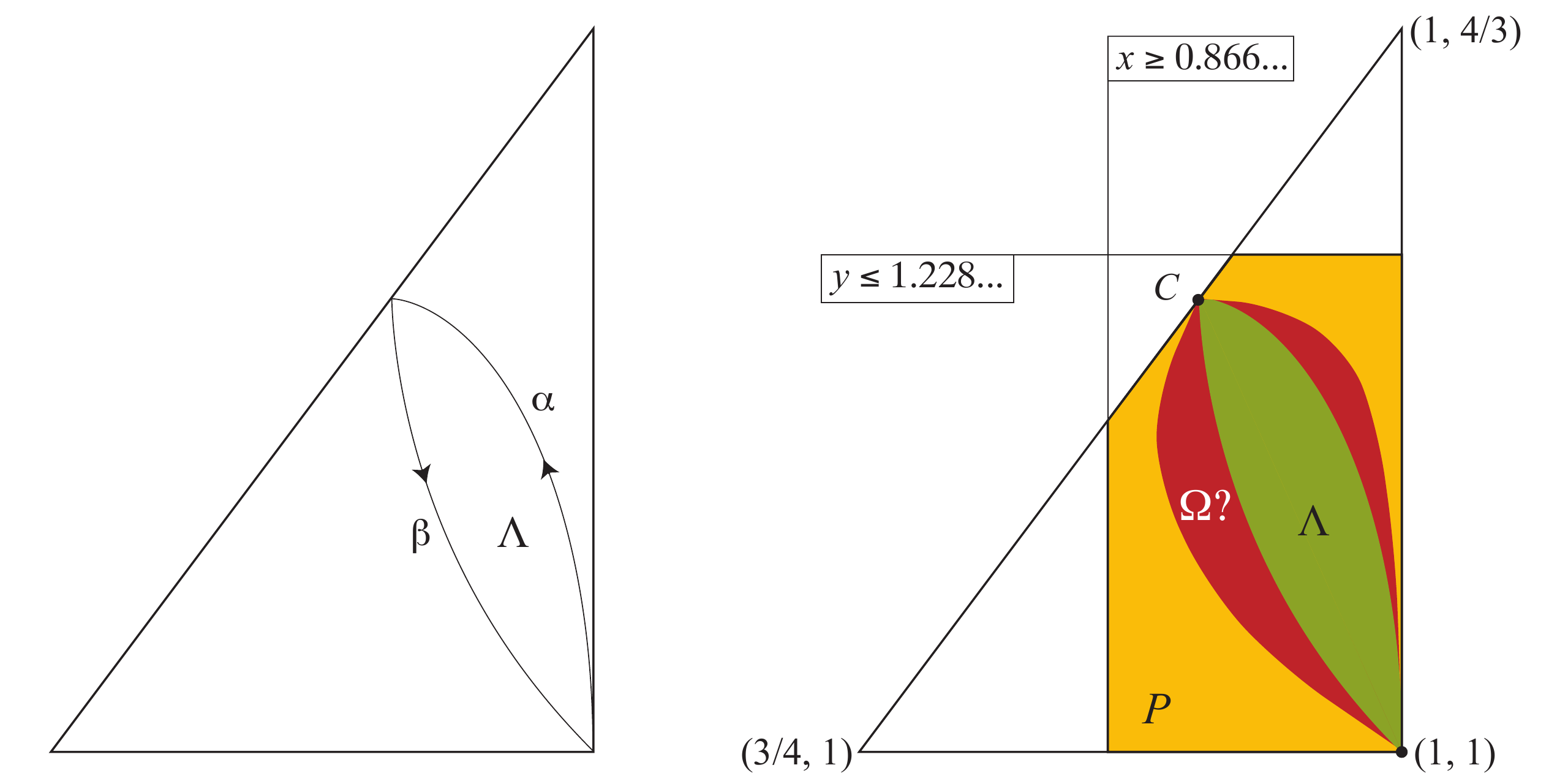}
\caption{The leaf $\Lambda$ and its position within $\Omega$.}
\label{leaf}
\end{center}
\end{figure}

Propositions 4 and 5 yield immediately the equalities
$$\delta(K_t) =\delta_T(K_t)=\delta_L(K_t),$$
$$\vartheta(K_t)=\vartheta_T(K_t)=\vartheta_L(K_t),$$
$$\delta(L_t)=\delta_T(L_t)=\delta_L(L_t),$$
and
$$\vartheta(L_t)=\vartheta_T(L_t)=\vartheta_L(L_t),$$
which imply that the leaf $\Lambda$ is contained in each of  $\Omega$, $\Omega_T$, and $\Omega_L$.

Moreover, since each of the disks $K_t$ and $L_t$ is centrally symmetric, the last observation of Remark 2, implies that $\Lambda$ is contained in each of $\Omega^*$, $\Omega^*_T$, and $\Omega^*_L$ as well.

\section{Questions}

\begin{enumerate}
\item[(Q1)] {\em Are the sets $\Omega$ and $\Omega^*$ compact?} {\bf Comment.} Since each of  the functions $\delta_T$, $\delta_L$, $\vartheta_T$ and $\vartheta_L$ is affine invariant (see Section 1, $(iii)$), and since the space $\left[\mathcal K^n \right]$ of affine
equivalence classes of all convex $n$-dimensional bodies is compact (see Macbeath \cite{McB51}), it follows that each of the sets $\Omega_T$, $\Omega_L$, and $\Omega^*_T = \Omega^*_L$ is compact.

\item[(Q2--Q4)] {\em Is each of the five $\Omega$-type sets simply-connected?  Is each of them a topological disk? Is each of them convex?}

\item[(Q5)]  {\em What are the vertical lines of support from the left for each of the five $\Omega$-type sets?}

\item[(Q6)] {\em What are the horizontal lines of support from above for each of the five $\Omega$-type sets?}

\item[(Q7)]  The inequalities of Ismailescu  \cite{DI01a, DI01b}
$$
\delta_L(K)+\vartheta_L(K)\ge2
$$
and
$$
\vartheta_L(K)\le1+\frac54\sqrt{1-\delta_L(K)}
$$
for every centrally symmetric convex disk $K$ can be expressed as: the set $\Omega^*_L=\Omega^*_T$ lies between the line $x+y=2$ and the curve $y=1+\sqrt{1-x}$. {\em Does the same hold for the set $\Omega$?}

In addition, the upper bound (2.7) on $\vartheta_T(K)$ for centrally symmetric convex disks further restricts the region $P_0$ containing $\Omega^*_L$, as shown in Fig.~\ref{dan}.

\end{enumerate}
\begin{figure}[h]
\includegraphics[scale=.5]{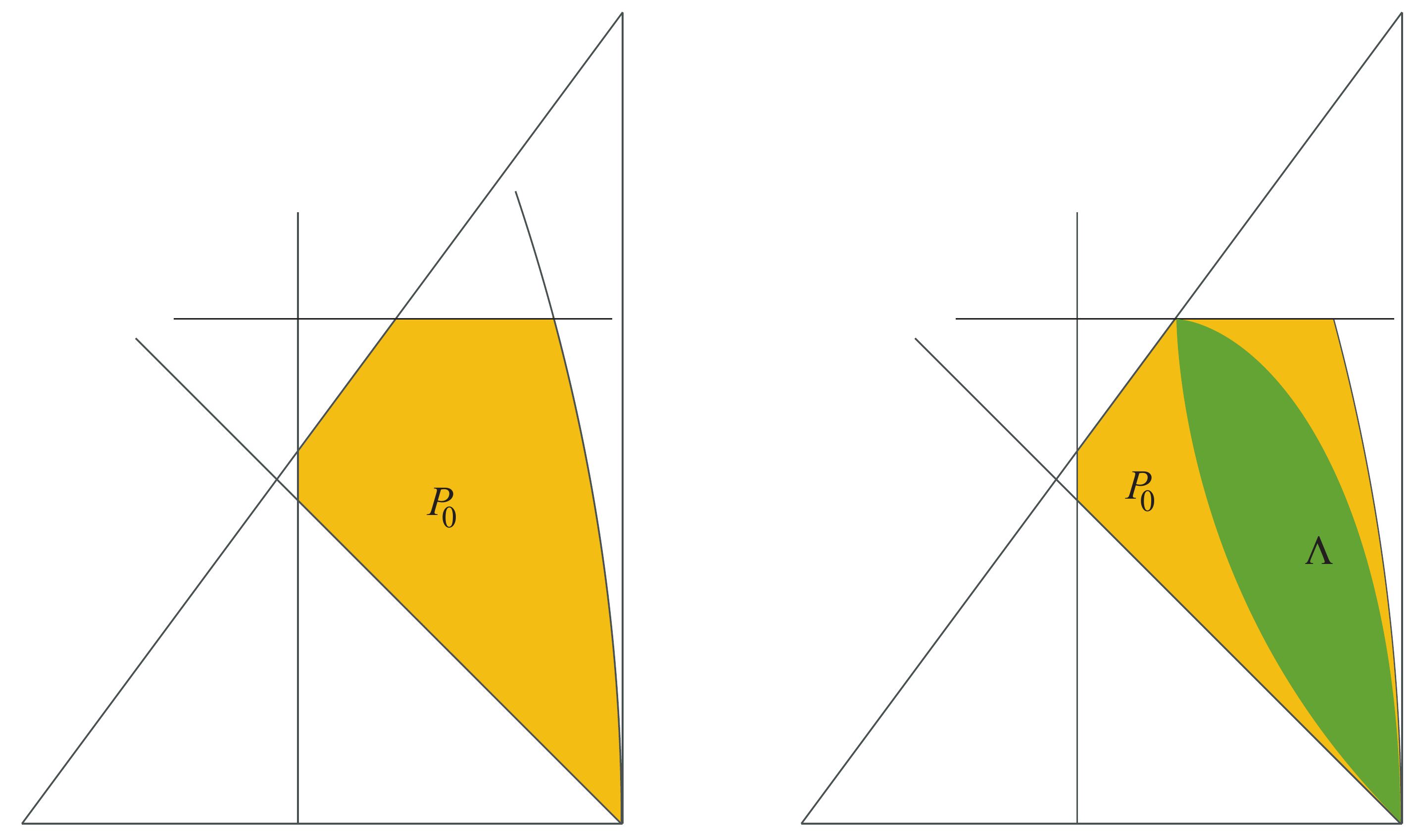}
\caption{The inequalities of Ismailescu bound the region $P_0$ containing $\Omega^*_L$; the leaf $\Lambda$ lies in it.}
\label{dan}
\end{figure}

The inequalities of Ismailescu were inspired by the numerical studies in his own doctoral thesis \cite{DI01a}, where he devised an algorithm to compute the maximum area of a hexagon contained in any given centrally symmetric octagon and the minimum area of a hexagon containing it.  The algorithm enabled him to compute $\delta(K)$ and $\vartheta(K)$ for every centrally symmetric octagon $K$ and he used it to plot points in $\Omega^*_L$ corresponding to a large number of randomly generated centrally symmetric octagons (see Figure \ref{octagons}(a)).  Moreover, based on the algorithm, he gave a complete, analytical description of the subset $U$ of $\Omega^*_L$ consisting of points that correspond to all centrally symmetric octagons, namely:
$$
U =\left\{(x,y)\in{\mathbb{R}}^2: 1\le y\le 4-2\sqrt2,\ \ \frac{5y^2-12y+8}{2y^2-5y+4}
\le x\le
\frac{y \left( y+4+\sqrt{y^2-8y+8}\,\right)}{4y+2}
 \right\},
$$
see Fig.~\ref{octagons}$(b)$.
\begin{figure}[h]
\includegraphics[scale=.6]{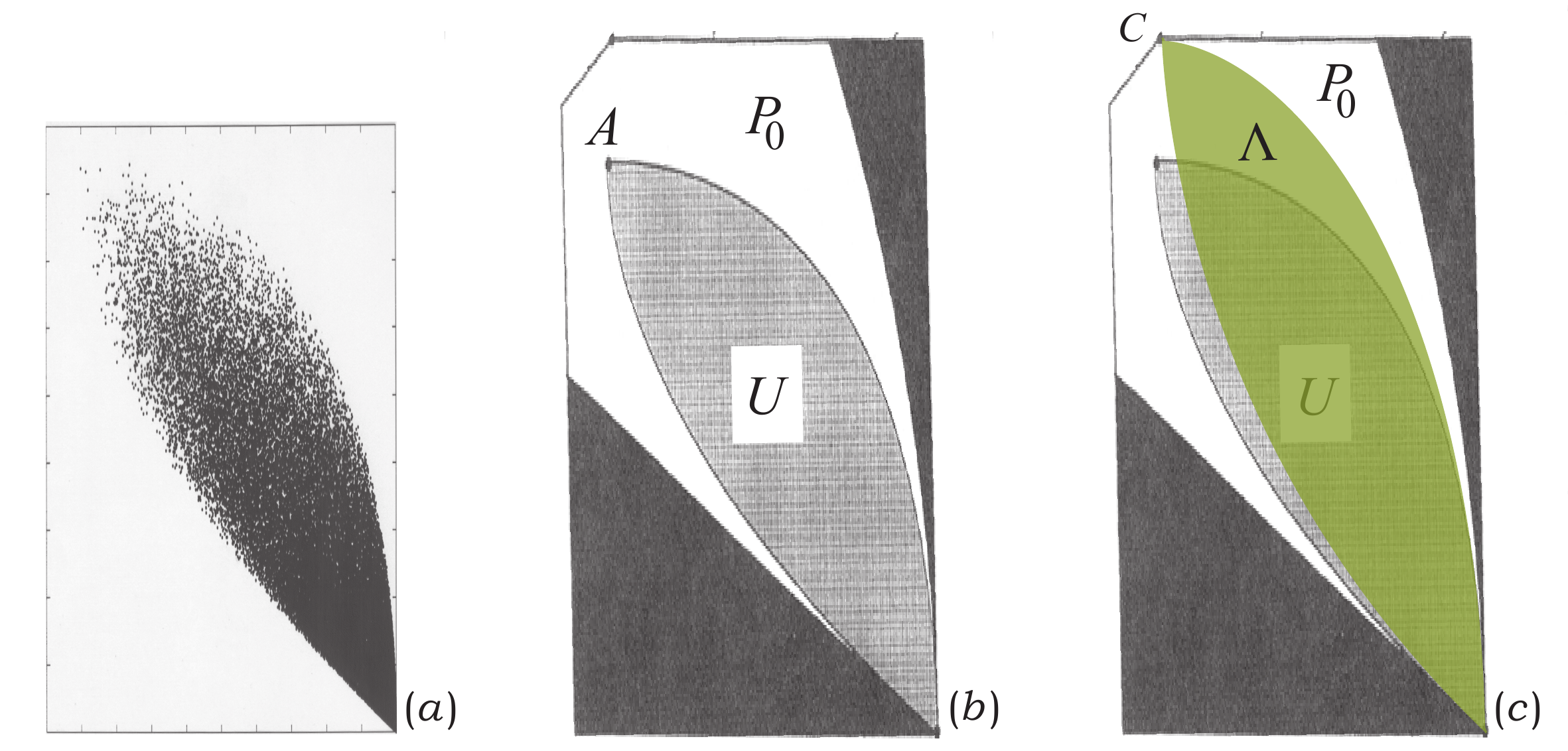}
\caption{The Ismailescu range $U$ and the leaf  $\Lambda$. (Diagrams used and modified by permission of Dan Ismailescu.)}
\label{octagons}
\end{figure}

Figure \ref{octagons} shows side-by-side: $(a)$---the plot of  50,000 points in $\Omega^*_L$ corresponding to random computer-generated centrally symmetric octagons; $(b)$---the analytically described range $U$, whose apex $A$ corresponds to the regular octagon; and $(c)$---the leaf-like disk $\Lambda$, described in Section 5, superimposed on the Ismailescu diagram; the apex $C$ of $\Lambda $ corresponds to the circular disk.

Note the unexpected resemblance between the shapes of $U$ and $\Lambda$. Also note that, somewhat surprisingly, a randomly generated centrally symmetric octagon seems quite likely to be close to a planar tile. {\em Is the same true for a randomly chosen $n$-gon with $n\ge 3$}?

A very recent result of Ismailescu and Kim \cite{IK13} includes the inequality $\delta_L(K)\vartheta_L(K)\ge1$ for every centrally symmetric $K$, which is stronger than Ismailescu's inequality $\delta_L(K)+\vartheta_L(K)\ge2$ mentioned above, further restricting the convex region $P_0$ containing $\Omega^*_L$ shown in Fig.~\ref{dan}.  Still, observe that at the left vertical edge of the region $P_0$ and at the upper-right corner of it (see Fig.~\ref{octagons}c) there seem to be further areas in $P_0$ free from points of $\Omega^*_L$, as if inviting discovery of additional restricting inequalities.

The questions of describing explicitly the sets $\Omega$, $\Omega_T$, $\Omega_L$, $\Omega^*$, and $\Omega^*_L$ remain open and appear to be extremely difficult. However, asking for some better approximations from outside and from inside seems reasonable.

\end{document}